\documentclass[a4paper, 12pt] {article}
\usepackage[T1,T2A]{fontenc}
\usepackage[cp1251]{inputenc}
\usepackage[russian,english]{babel}
\usepackage{amsfonts}
\usepackage{mathtext}
\usepackage{hyperref}
\usepackage{cite}

\usepackage{graphicx}

\usepackage{verbatim}
\usepackage{amsmath}
\usepackage{amsthm}
\usepackage{amssymb}
\usepackage{delarray}
\usepackage{mathrsfs}
\usepackage{xcolor}

\begin{document}
\begin{center}
{\bf\Large On the problem of recovery of Sturm--Liouville operator with two frozen arguments}
\end{center}

\begin{center}
{\bf\large Maria Kuznetsova\footnote{Department of Mathematics, Saratov State University, e-mail: {\it kuznetsovama@info.sgu.ru}}}
\end{center}

\noindent{\bf Abstract.}
{Inverse spectral problems consist in recovering operators by their spectral charac\-teristics.
The problem of recovering the Sturm--Liouville operator with one frozen argument
was studied earlier in works of various authors.
In this paper, we study a uniqueness of recovering operator with two frozen arguments and different coefficients $p, q$ by the spectra of two boundary value problems.
The case considered here is significantly more difficult than the case of one frozen argument, because the operator is no more a one-dimensional perturbation.
We prove that the operator with two frozen arguments, in general case, can not be recovered by the two spectra.
For the uniqueness of recovering, one should impose some conditions on the coefficients.
We assume that the coefficients $p$ and $q$ equal zero on certain segment and prove a uniqueness theorem.
As well, we obtain regularized trace formulae for the two spectra.
The result is formulated in terms of convergence of certain series, which allows us to avoid restrictions on the smoothness of the coefficients.}


\\

\newtheorem{lemma}{Lemma}[section]
\newtheorem{theorem}{Theorem}[section]

\theoremstyle{definition}
\newtheorem{example}{Example}[section]
\newtheorem{problem}{Inverse problem}[section]

\renewcommand{\leq}{\leqslant}
\renewcommand{\geq}{\geqslant}

\numberwithin{equation}{section}
%

\section{Introduction}
In this paper, we consider an inverse spectral problem for a Sturm--Liouville operator with two frozen arguments:
 $$\ell y =  -y''(x) + p(x) y(a) + q(x) y(b), \quad x \in (0, \pi),$$
where $p,q \in L_2(0, \pi)$ are complex-valued, while the parameters $a, b \in (0, \pi)$ are fixed and called frozen arguments. Unlike purely differential operators, studied in the classical theory of inverse spectral problems~\cite{KuM_march,KuM_B,KuM_levitan,KuM_yurko,KuM_high-order}, the operator $\ell y$ is nonlocal.
Nonlocal operators have special spectral properties, see e.g.~\cite{KuM_albeverio, KuM_bolletin,malyug,KuM_nonlocal mc,vlad,con,KuM_nizh10}, and require the development of methods other than the methods of the classical theory of inverse spectral problems.

In the previous works~\cite{KuM_DobHry,KuM_AML,KuM_my-uniform,mmi,
KuM_BV,KuM_BBV,KuM_BK,KuM_albeverio,KuM_nizh10,KuM_wang,KuM_but-hu,
KuM_tsai,KuM_my-froz,KuM_bon-disc}, there were studied the Sturm--Liouville operators with one frozen argument, i.e. when $q=0,$ under various boundary conditions.
The following statement of an inverse spectral problem was considered: given the operator spectrum, recover the coefficient $p.$
The most complete results were obtained for the boundary conditions
 $y^{(\alpha)}(0) = y^{(\beta)}(\pi) = 0,$ where $\alpha,\beta \in \{0,1\}$ denote the order of derivative.
In the rational case $a/\pi \in \mathbb Q,$ certain part of the spectrum may degenerate, i.e. it does not depend on $p,$ and for uniqueness of recovering $p,$ besides the spectrum, additional data is needed, see~\cite{KuM_BV,KuM_BBV,KuM_BK}.
In the irrational case $a/\pi \notin \mathbb Q,$ the degeneration effect does not occur,
and $p$ is uniquely recovered by the spectrum, see~\cite{KuM_wang}. Thus, the uniqueness of recovering one coefficient by one spectrum takes place for a.e. $a \in (0, \pi).$

A general approach to both cases was developed in the works~\cite{KuM_AML,KuM_my-uniform}. Later on, it was generalized on operators of the form $\tilde \ell y = -y''(x) + p(x) y(a) + r(x) y(x),$ see~\cite{mmi}.
In the paper~\cite{KuM_DobHry}, it was noticed that the Sturm--Liouville operator with one frozen argument is a one-dimensional perturbation of the differential operator $\ell_0 y =-y''$ and an inverse problem for the corresponding class of one-dimensional perturbations was studied.
Recently, there appeared an interest to operators with several frozen arguments taken with the same coefficient, see~\cite{preprint}:
 $$\ell_1 y = -y''(x) + \sum_{k=1}^m y(a_k) p(x), \quad m \in \mathbb N.$$
These operators are one-dimensional perturbations of the type that was studied in the work~\cite{KuM_DobHry}.
Their consideration does not lead to a situation significantly different from the case of one frozen argument, and inverse spectral problems can be studied by the methods of the works~\cite{KuM_DobHry,KuM_AML,KuM_my-uniform,mmi}.
At the same time, the operator of interest $\ell y$ is not a one-dimensional perturbation, and for it, there are no known methods of the theory of inverse spectral problems.

Introduce BVP ${\mathcal L}_j(p, q)$ with the number  $j=0,1:$
\begin{equation} \label{eq with 2 froz}
\ell y = \lambda y(x),
\end{equation}
\begin{equation} \label{bc}
y^{(j)}(0) = y(\pi) = 0,
\end{equation}
and denote by $\{ \lambda_{nj}\}_{n \ge 1}$ its spectrum.
Consider the following inverse problem:
 \begin{problem} \label{first ip}
Given the spectra  $\{ \lambda_{n0}\}_{n \ge 1}$ and $\{ \lambda_{n1}\}_{n \ge 1},$ recover  $p$ and $q.$
\end{problem}
\noindent
First of all, we are interested in the uniqueness of the solution of this inverse problem,
i.e. whether different pairs $(p, q)$ always correspond to the different pairs of spectra $(\{ \lambda_{n0}\}_{n \ge 1},$ $\{ \lambda_{n1}\}_{n \ge 1}).$
We note at once that if $a=b,$ then only the sum of $p$ and $q$ matters, and in this case, the solution of Inverse problem~\ref{first ip} is not unique.
To exclude this situation, we impose the condition
\begin{equation} \label{ab}
0< a < b < \pi.
\end{equation}
We obtain that for any $a$ and $b$ satisfying~\eqref{ab} the solution of Inverse problem~\ref{first ip} is not unique.
For the uniqueness of recovering $p$ and $q,$ one should clarify the statement of the inverse problem by specifying additional information.
We assume that the coefficients $p$ and $q$ both equal zero on  $[0, b]$ or $[a, \pi].$
Under this condition, we prove the theorem on uniqueness of recovering $p$ and $q$ by the two spectra.

We also obtain regularized trace formulae for the spectra $\{ \lambda_{nj}\}_{n \ge 1},$ $j=0,1.$
A regularized trace is understood as the series of the differences between the eigenvalues of two operators, one of which is a perturbation of the other.
From a physical point of view, this notion reflects the measure of the energy defect of a quantum system, see~\cite{Lif}.
Basic results from the theory of regularized traces are given in the review~\cite{Sad-Pod}.

Regularized trace formulae for the Sturm--Liouville operators with one frozen argument were considered in the works~\cite{huang, bon-traces,dob-traces,sat-traces} in the case of the absolutely continuous coefficient $p.$
In~\cite{bon-traces}, under the conditions $ p \in W^1_2[0, \pi]$ and $q \equiv 0,$ it was proved that
$$\sum_{n=1}^\infty  \left(\lambda_{nj} - \Big(n - \frac{j}{2}\Big)^2\right) = p(a), \quad j=0,1.$$
We obtain regularized trace formulae under more general conditions
 $p, q \in L_1[0, \pi]:$
\begin{equation} \label{traces}
\sum_{n=1}^\infty \left(\lambda_{nj} - \Big(n - \frac{j}{2}\Big)^2\right)=\sum_{n=1}^\infty s_{nj}, \quad j=0,1,
\end{equation}
where $s_{nj}$ are expressed from the Fourier coefficients of the functions $p$ and $q$ by the system of the eigenfunctions of the unperturbed operator $-y''.$
Formula~\eqref{traces} is understood in such way that either both series diverge or converge to the same number. If $p$ is absolutely continuous at the vicinity of the point $a$ and $q$  is absolutely continuous at the vicinity of the point $b,$ then convergence to the number $p(a) + q(b)$ takes place.
Regularized trace formulae for the operators with two frozen arguments can be also obtained from the results of the works~\cite{Lyub, cop} but stronger restrictions of the coefficients are required.

The paper is organized as follows. In Section~\ref{char}, we construct characteristic functions and obtain asymptotic formulae for the spectra, see Theorem~\ref{th asympt}.
In Section~\ref{non-unique}, we construct the distinct pairs $(p, q)$
that lead to the same pair of the spectra, see Theorem~\ref{th main}.
In Section~\ref{sec ip}, we provide the clarified statement of the inverse problem and prove the uniqueness theorem for it, see Theorem~\ref{uniqueness theorem}.
In Section~\ref{traces sec}, we obtain the regularized trace formulae for the spectra
 $\{ \lambda_{nj}\}_{n \ge 1},$ see Theorem~\ref{th traces}.
In Appendix, we give details of the proof of formula~\eqref{rep} for the characteristic functions.

\section{Characteristic functions} \label{char}
Let us obtain characteristic functions of the BVPs ${\mathcal L}_j(p, q)$ for $j=0,1.$
It is well known that for $f \in L_2(0, \pi),$ any solution of the equation $-y''(x) + f(x)= \lambda y(x)$ can be represented in a form
\begin{equation*}
y(x) = x_1 \frac{\sin \rho x}{\rho} + x_2 \cos \rho x + \int_0^x \frac{\sin \rho(x-t)}{\rho} f(t) \, dt, \quad \rho^2 = \lambda, \quad x_1, \, x_2 \in \mathbb C.
\end{equation*}
Putting $f(t) =  x_3 p(t) + x_4 q(t),$ we get
\begin{equation} \label{solution}
y(x) = x_1 \frac{\sin \rho x}{\rho} + x_2 \cos \rho x + x_3  \int_0^x \frac{\sin \rho(x-t)}{\rho} p(t) \, dt + x_4  \int_0^x \frac{\sin \rho(x-t)}{\rho} q(t) \, dt.
\end{equation}
The latter function is a solution of equation~\eqref{eq with 2 froz} if and only if $y(a)=x_3$ and ${y(b)=x_4}.$
For this solution, to be non-trivial, it is necessary and sufficient that the vector $(x_j)_{j=1}^4$ is non-zero.

Substituting expression~\eqref{solution} into conditions~\eqref{bc} for $j=0,1$  and into equalities $y(a) = x_3,$ $y(b) = x_4,$ we arrive at a system of linear equations
\begin{equation}
\left\{
\begin{array}{c}
x_{2-j} = 0,\\[4mm]
\displaystyle x_1 \frac{\sin \rho \pi}{\rho} + x_2 \cos \rho \pi + x_3 \int_0^\pi \frac{\sin \rho(\pi - t)}{\rho} p(t) \, dt + x_4 \int_0^\pi \frac{\sin \rho(\pi - t)}{\rho} q(t) \, dt = 0, \\[4mm]
\displaystyle x_1 \frac{\sin \rho a}{\rho} + x_2 \cos \rho a  + x_3 \int_0^a \frac{\sin \rho(a - t)}{\rho} p(t) \, dt + x_4 \int_0^a \frac{\sin \rho(a - t)}{\rho} q(t) \, dt = x_3, \\[4mm]
\displaystyle x_1 \frac{\sin \rho b}{\rho} + x_2 \cos \rho b + x_3 \int_0^b \frac{\sin \rho(b - t)}{\rho} p(t) \, dt +
x_4 \int_0^b \frac{\sin \rho(b - t)}{\rho} q(t) \, dt = x_4.
\end{array} \right.
\end{equation}
The system has a non-zero solution $(x_k)_{k=1}^4$ if and only if
$$\Delta_j(\lambda) :=  \left|
\begin{array}{ccc}
\varphi_j(\rho, \pi) & \displaystyle \int_0^\pi \frac{\sin \rho(\pi - t)}{\rho} p(t) \, dt & \displaystyle \int_0^\pi \frac{\sin \rho(\pi - t)}{\rho} q(t) \, dt \\[4mm]
\varphi_j(\rho, a) & \displaystyle\int_0^a \frac{\sin \rho(a - t)}{\rho} p(t) \, dt - 1 & \displaystyle\int_0^a \frac{\sin \rho(a - t)}{\rho} q(t) \, dt \\[4mm]
\varphi_j(\rho, b) & \displaystyle\int_0^b \frac{\sin \rho(b - t)}{\rho} p(t) \, dt & \displaystyle\int_0^b \frac{\sin \rho(b - t)}{\rho} q(t) \, dt -1
\end{array} \right| = 0,
$$
where we designate
$\varphi_j(\rho, z) = \left\{ \begin{array}{cc}
\rho^{-1}\sin \rho z, & j=0, \\
 \cos \rho z, & j=1.
\end{array} \right.
 $
 From now on, we denote the dependence on $p$ and $q$ by a set of arguments after semicolon, for example, $\Delta_j(\lambda)=\Delta_j(\lambda; p, q).$
We may not specify this dependence if there is no emphasis on the specific values of $p$ and $q.$

For $j=0,1,$ the function $\Delta_j(\lambda)$ is a characteristic function of the BVP ${\mathcal L}_j(p, q):$ the zeros of this function coincide with the spectrum of the BVP. Since the Taylor series of the entire functions $\rho^{-1}\sin \rho z$ and $\cos \rho z$ contain $\rho$ only in even powers, the functions $\Delta_0$ and $\Delta_1$ are entire functions of $\lambda.$

Expanding the determinants, after manipulations, we obtain representations
\begin{equation} \label{rep}
\Delta_j(\lambda) =  \varphi_j(\rho, \pi) + A_{j0}(\lambda)+ A_{j1}(\lambda) + B_j(\lambda), \quad  j=0,1,
\end{equation}
where
\begin{equation}
\left.\begin{array}{c}
\displaystyle A_{j0}(\lambda) =  A_{j0}(\lambda; p) = \varphi_j(\rho, a) \int_a^\pi \frac{\sin \rho(\pi - t)}{\rho} p(t) \,dt + \frac{\sin \rho (\pi - a)}{\rho} \int_0^a \varphi_j(\rho, t) \, p(t) \,dt, \\[4mm]
\displaystyle A_{j1}(\lambda) =  A_{j1}(\lambda; q)=\varphi_j(\rho, b) \int_b^\pi \frac{\sin \rho(\pi - t)}{\rho} q(t) \,dt + \frac{\sin \rho (\pi - b)}{\rho} \int_0^b \varphi_j(\rho, t) \,q(t) \,dt,
\end{array} \right\}
\label{def A}
\end{equation}
\begin{multline}
B_j(\lambda) =  B_j(\lambda; p, q) =
\frac{\sin \rho(\pi - b)}{\rho} \Big( \int_0^a \varphi_j(\rho, t) \, p(t) \, dt  \int_a^b \frac{\sin \rho (\xi - a)}{\rho} q(\xi) \, d\xi -\\
\hspace{7cm}- \int_0^a \varphi_j(\rho, \xi) \, q(\xi) \, d\xi  \int_a^b \frac{\sin \rho (t - a)}{\rho} p(t) \, dt\Big) + \\
 + \frac{\sin \rho(b-a)}{\rho} \Big( \int_0^a \varphi_j(\rho, t)\, p(t) \, dt \int_b^\pi \frac{\sin \rho(\pi - \xi)}{\rho} q(\xi)\, d\xi- \hspace{5cm} \\
\hspace{7cm}  -\int_0^a \varphi_j(\rho, \xi)\, q(\xi) \, d\xi \int_b^\pi \frac{\sin \rho(\pi - t)}{\rho} p(t)\, dt\Big) +\\
+ \varphi_j( \rho, a) \Big( \int_b^\pi \frac{\sin \rho(\pi - \xi)}{\rho} q(\xi)\, d\xi \int_a^b \frac{\sin \rho(b-t)}{\rho} p(t)\, dt  - \hspace{5cm}\\
\hspace{7cm}-\int_b^\pi \frac{\sin \rho(\pi - t)}{\rho} p(t)\, dt \int_a^b \frac{\sin \rho(b-\xi)}{\rho} q(\xi)\, d\xi \Big)+ \\
+\varphi_j( \rho, a)\frac{\sin \rho(\pi - b)}{\rho}\int_a^b \int_a^b \frac{\sin \rho(\xi - t)}{\rho} q(\xi) \,p(t) \,dt\,d\xi
\label{Bj}
\end{multline}
(for details, see Appendix).
From formula~\eqref{rep} it is clear that $\Delta_j(\lambda)$ are entire functions of order $1/2$ and type $\pi.$
By the standard method, based on application of Rouche's theorem  (см.~\cite{KuM_yurko}), we have proved the following theorem.

\begin{theorem} \label{th asympt}
For $j=0,1,$ asymptotic formulae
$$ \lambda_{nj} = \Big(n - \frac{j}{2}\Big)^2 + \varkappa_{nj}, \quad \{ \varkappa_{nj}\}_{n \ge 1} \in \ell_2,$$
hold.
\end{theorem}
In what follows, certain properties  of the terms in representation~\eqref{rep}
will play an important role. The terms $A_{j0}(\lambda; p)$ and $A_{j1}(\lambda; q)$ depend linearly on $p$ and $q,$ respectively. The terms $B_j( \lambda; p, q)$ depend bilinearly on $p$ and $q.$ From formula~\eqref{Bj} it is clear that this dependence is antisymmetric: $B_j(\lambda; p, q) = -B_j(\lambda; q, p)$ (pairs of antisymmetric terms are combined into brackets in~\eqref{Bj}). By this reason, if $p=q,$ then $B_j(\lambda; p, q) = 0.$ The bilinearity yields that
\begin{equation} \label{anti}
B_j(\lambda;  q, \alpha q) = 0, \quad \alpha \in \mathbb C.
\end{equation}
From~\eqref{Bj} one can also see that $B_j(\lambda;  p, q) = 0$ whenever
\begin{equation} \label{Ba}
\left[\begin{array}{cc}
p(x) \equiv 0, \; q(x) \equiv 0, & x\in [a, \pi],\\[2mm]
 p(x) \equiv 0, \; q(x) \equiv 0, & x\in [0, b].
 \end{array}\right.
\end{equation}

Construction of the characteristic function of the BVP ${\mathcal L}_0$  was considered in~\cite{preprint} in the particular case $p=q.$ In this case $B_j(\lambda) =  0,$ and our representation~\eqref{rep} agrees with the formula obtained in~\cite{preprint}.
Another case in which $B_j(\lambda) = 0$ is when $q=0.$ The latter condition also yields $A_{j1}(\lambda)=0,$ and formula~\eqref{rep} gives us the representation for the characteristic function of the operator with one frozen argument, which agrees with the one obtained earlier, see~\cite{KuM_BK,KuM_BV,KuM_BBV}.


\section{Non-uniqueness of solution of Inverse problem~\ref{first ip}}
\label{non-unique}
In this section, we construct the distinct pairs of the coefficients
 $(p, q)$ that lead to the same pairs of the spectra $\Big( \{ \lambda_{n0}\}_{n \ge 1}, \{ \lambda_{n1}\}_{n \ge 1}\Big).$ By this way
we prove non-uniqueness of solution of Inverse problem~\ref{first ip}.

We extend the functions $p$ and $q$ on $\mathbb R \setminus (0, \pi)$ by zero.
Proceeding analogously to the proof of Lemma~1 in~\cite{KuM_BK}, we get representations
\begin{equation}\label{Aj}
\left.
\begin{array}{c}
\displaystyle A_0(\lambda; p, q):= A_{00}(\lambda)+ A_{01}(\lambda) =  \frac12\int_0^\pi \frac{\cos\rho t}{\rho^2} W_0(t) dt, \\[4mm]
\displaystyle A_1(\lambda; p, q) :=  A_{10}(\lambda)+A_{11}(\lambda) =  \frac12\int_0^\pi \frac{\sin\rho t}{\rho}W_1(t) dt,
\end{array} \right\}
\end{equation}
where
\begin{multline*}
W_j(t; p, q) = (-1)^{j+1}p(t+a - \pi) + (-1)^{j+1} p(\pi - t + a)  +\\[1mm]
+ (-1)^j p(\pi - a+t)  + p(\pi - a - t) + \\[1mm]
+ (-1)^{j+1}q(t+b - \pi) + (-1)^{j+1} q(\pi - t + b)  +\\[1mm]
+ (-1)^j q(\pi - b+t) + q(\pi - b - t), \quad t \in [0, \pi], \quad j=0,1.
\end{multline*}

\begin{lemma} \label{l A=0}
Put $T =
\min\{ a, b-a, \pi-b\}$
and let $G(t) \in L_2(\mathbb R)$ be an arbitrary even non-trivial function that equals zero outside the segment
$[-T, T].$ Then, for functions
\begin{equation} \label{sr}
s(t) = G(b - t),\quad
r(t) =
-G(a - t)
\end{equation}
we have $A_j(\lambda; s, r) = 0,$ $j=0,1.$
\end{lemma}
\begin{proof}
For $t \in [0, \pi],$ introduce functions
\begin{equation} \label{u0}
\left.\begin{array}{c}
 \displaystyle
 u_0(t; p, q) := \frac{W_0(t) + W_1(t)}{2} = p(\pi - a -t) + q(\pi - b -t), \\[3mm]
 \displaystyle
 u_1(t; p, q) := \frac{W_0(t) - W_1(t)}{2} = -p(t+a - \pi) - p(\pi - t + a)
 + p(\pi - a+t) - \\[3mm]
   \hspace{7cm}
  -q(t+b - \pi)    - q(\pi - t + b)  + q(\pi - b+t).
\end{array} \right\}
\end{equation}
Note that $s$ and $r$ equal $0$ outside $[0, \pi],$
so in~\eqref{u0} we can replace $p$ with $s$ and $q$ with $r.$
Let us prove that $u_0(t; s, r) = u_1(t; s, r)=0.$

Indeed, $$u_0(t; s, r) = s(\pi - a -t) + r(\pi - b -t)\overset{\eqref{sr}}{=}
G(b-\pi + a+t) - G(a - \pi + b + t) = 0.$$
In $u_1(t; s, r)$ we group the terms as follows:
\begin{multline*}u_1(t; s, r) =  -\big(s(t+a - \pi)  + r(\pi - t + b) \big) -\\ - \big(s(\pi - t + a)  +r(t+b - \pi) \big) +  \big(s(\pi - a+t) + r(\pi - b+t) \big). \end{multline*}
Applying~\eqref{sr} to each pair of terms and using the evenness of $G(t),$ we arrive at the equality $u_1(t; s, r) = 0.$

From $u_0(t; s,r) = u_1(t; s, r) = 0$ it follows that $W_0(t; s, r) = 0$ and $W_1(t; s, r) = 0.$ By~\eqref{Aj}, we obtain $A_j(\lambda; s,r) =0$ for $j=0,1.$
\end{proof}

\begin{theorem}\label{th main}
The distinct pairs of the coefficients $(p, q) = (-r, s+r)$ и $(p, q) = (-s-r, s)$ lead to the same pair of spectra $\Big( \{ \lambda_{n0}\}_{n \ge 1}, \{ \lambda_{n1}\}_{n \ge 1}\Big).$ Thus, Inverse problem~\ref{first ip} has no unique solution.
\end{theorem}

\begin{proof} Since $G$ in Lemma~\ref{l A=0} is a non-trivial function, $s$ and $r$ are also non-trivial functions, and $(-r, s+r)\ne(-s-r, s).$
Let $j=0,1.$ By the linearity of $A_{j0}(\lambda; p)$ and $A_{j1}(\lambda; q)$ with respect to $p$ and $q,$ we have
\begin{multline*}
A_j(\lambda;-r, s+r) - A_j(\lambda;-s-r, s)
=A_{j0}(\lambda;-r)+ A_{j1}(\lambda; s+r)-  \\-A_{j0}(\lambda;-s-r) -A_{j1}(\lambda;s)
=A_{j0}(\lambda; s) + A_{j1}(\lambda; r)= A_j(\lambda; s, r)=0,\end{multline*}
where the last equality holds due to Lemma~\ref{l A=0}. Thus, $A_j(\lambda;-r, s+r) = A_j(\lambda;-s-r, s).$

Using property~\eqref{anti} and the bilinearity of $B_j(\lambda; p, q)$ with respect to $p$ and $q,$ we obtain
$$B_j(\lambda;-r,s+r) = B_j(\lambda;-r,s) = B_j(\lambda;-s-r, s).$$
By formula~\eqref{rep}, we have $\Delta_j(\lambda;-r, s+r) = \Delta_j(\lambda;-s-r, s).$ This means that the pairs $(p, q) = (-r, s+r)$ and $(p, q) = (-s-r, s)$ give the same spectrum $\{ \lambda_{nj}\}_{n \ge 1}.$
\end{proof}
Choosing a specific function $G$ satisfying the conditions of Lemma~\ref{l A=0}, we obtain certain pairs of the coefficients $(p, q)$ in Theorem~\ref{th main}.
\begin{example}
Let $a=\frac\pi4$ and $b=\frac\pi2.$ Then, the function $$G(t) = \chi_{[-T, T]}(t), \quad T = \frac\pi4, \quad \chi_S(t) :=
\left\{ \begin{array}{cc}
1, & t \in S, \\
0, & t \notin S,
\end{array}\right.$$
satisfies the conditions of Lemma~\ref{l A=0}.
We get $s(t) = \chi_{[\frac\pi4;\frac{3\pi}{4}]}(t)$ and $ r(t) =-\chi_{[0;\frac{\pi}{2}]}(t).$
By Theorem~\ref{th main}, the following pairs of the coefficients $(p, q)$ lead to the same pair of the spectra:
$$ p(t) =  \chi_{[0;\frac{\pi}{2}]}(t) , \quad q(t) = \chi_{[\frac\pi4;\frac{3\pi}{4}]}(t) - \chi_{[0;\frac{\pi}{2}]}(t) ;$$
$$ p(t) = - \chi_{[\frac\pi4;\frac{3\pi}{4}]}(t)+\chi_{[0;\frac{\pi}{2}]}(t), \quad q(t) = \chi_{[\frac\pi4;\frac{3\pi}{4}]}(t).$$
Direct substitution of each pair $(p, q)$ into formulae~\eqref{rep}--\eqref{Bj} confirms that the characteristic functions coincide:
\begin{equation*} \begin{array}{c}
\displaystyle
\Delta_0(\lambda; p, q) = \frac{1}{\rho} \sin \rho \pi + \frac{1}{2\rho^3}\Big( -3 \sin \rho\pi + 5\sin \frac{3\rho \pi}{4}  - 2 \sin \frac{\rho\pi}{2}+  \sin \frac{\rho\pi}{4}\Big) + \quad \quad   \quad \quad \quad \quad \\[3mm]
\displaystyle \quad \quad \quad \quad \quad \quad \quad \quad
+ \frac{1}{\rho^5} \Bigg( \sin \frac{\rho \pi}{2} \Big[\cos \frac{\rho \pi}{4} - 1\Big]^2 + 2 \sin \frac{\rho \pi}{4} \Big[\cos \frac{\rho \pi}{4} - 1\Big] \Big[\cos \frac{\rho \pi}{2} - \cos \frac{\rho \pi}{4}\Big] \Bigg),
\\[4mm]
\displaystyle\Delta_1(\lambda; p, q) = \cos \rho \pi +\frac{1}{2\rho^2}\Big( - 3\cos \rho\pi + 3 \cos \frac{3\rho \pi}{4} + \cos \frac{\rho\pi}{4} - 1\Big) + \quad \quad   \quad \quad \quad \quad \quad \quad\quad \quad\\[3mm]
\displaystyle  \quad \quad\quad \quad+ \frac{1}{\rho^4} \Bigg( \sin \frac{\rho\pi}{2} \sin \frac{\rho\pi}{4} \Big[1 - \cos  \frac{\rho\pi}{4}\Big] + \sin^2 \frac{\rho\pi}{4} \Big[\cos \frac{\rho\pi}{4} - \cos \frac{\rho\pi}{2}\Big] +\\[3mm]
\displaystyle \quad \quad   \quad \quad \quad \quad \quad \quad\quad \quad
\quad \quad   \quad \quad\quad \quad   \quad \quad\quad \quad   \quad \quad
+\cos \frac{\rho\pi}{4} \Big[\cos \frac{\rho\pi}{4} - \cos \frac{\rho\pi}{2}\Big] \Big[1 - \cos  \frac{\rho\pi}{4}\Big]\Bigg).
\end{array}
\end{equation*}
\end{example}

\section{Inverse problem with additional conditions}
\label{sec ip}

We consider Inverse problem~\ref{first ip} under additional conditions on $p$ and $q:$
\begin{problem} \label{second ip}
It is known that the pair of the coefficients $(p, q)$ satisfies conditions~\eqref{Ba}. Given spectra $\{ \lambda_{n0}\}_{n \ge 1}$ and $\{ \lambda_{n1}\}_{n \ge 1},$ recover $p$ and $q.$
\end{problem}

Now, we prove a theorem on the uniqueness of solution of Inverse problem~\ref{second ip}.
In addition to the BVPs ${\mathcal L}_0(p, q)$ and ${\mathcal L}_1(p, q),$ we consider   the BVPs ${\mathcal L}_0(\tilde p, \tilde q)$ and ${\mathcal L}_1(\tilde p, \tilde q)$ with other coefficients $\tilde p, \tilde q \in L_2(0, \pi).$
For $j=0, 1,$ denote by $\{\tilde\lambda_{nj} \}_{n \ge 1}$ the spectrum of the BVP ${\mathcal L}_j(\tilde p, \tilde q).$

\begin{theorem} \label{uniqueness theorem}
Let the functions $p,$ $q,$ $\tilde p,$ and $\tilde q$ satisfy one of the following two conditions:
\begin{enumerate}
\item Each function is zero on $[a, \pi];$
\item Each function is zero on $[0, b].$
\end{enumerate}
Then, the equalities $\{ \lambda_{n0} \}_{n \ge 1}=\{ \tilde\lambda_{n0} \}_{n \ge 1}$ and $\{ \lambda_{n1} \}_{n \ge 1}=\{ \tilde\lambda_{n1} \}_{n \ge 1}$ yield $p = \tilde p$ and $q = \tilde q.$
\end{theorem}

We need the following lemma.

\begin{lemma} \label{char lemma}
The characteristic functions are uniquely recovered by the spectra:
$$\Delta_j(\lambda) = \pi^{1-j}\prod_{k=1}^\infty \frac{\lambda_{nj} - \lambda}{(n - \frac{j}{2})^2}, \quad j=0,1.$$
\end{lemma}
The proof of the lemma is standard, see the proof of Theorem~1.1.4 in~\cite{KuM_yurko}. In the proof, we applied the asymptotic formulae from Theorem~\ref{th asympt} and the formulae
$$\Delta_0(\lambda) = \frac{\sin \rho \pi}{\rho} + O\Big(\frac{e^{|\tau|\pi}}{\rho^2}\Big), \quad \Delta_1(\lambda) = \cos \rho \pi + O\Big(\frac{e^{|\tau|\pi}}{\rho}\Big), \quad \tau = \mathrm{Im}\, \rho,$$
which follow from~\eqref{rep}--\eqref{Bj}.
\begin{proof}[Proof of Theorem~\ref{uniqueness theorem}.]
By virtue of Lemma~\ref{char lemma}, we have that $\Delta_j(\lambda; p, q) = \Delta_j(\lambda; \tilde p, \tilde q),$ $j=0,1.$ Since $B_j(\lambda; p, q) = B_j(\lambda; \tilde p, \tilde q) =0,$ in representation~\eqref{rep}, we get $A_j(\lambda; p, q) = A_j(\lambda; \tilde p, \tilde q),$ and in~\eqref{Aj}, we obtain $W_j(t; p, q) = W_j(t; \tilde p, \tilde q),$ $j=0,1.$
This yields
\begin{equation} \label{prev eq}
u_j(t; p, q) = u_j(t; \tilde p, \tilde q), \quad j=0,1.
\end{equation}
Denote $\hat p = p - \tilde p$ and $\hat q = q - \tilde q.$ From~\eqref{u0} and~\eqref{prev eq}, it follows that
\begin{equation} \label{u01}
\left.\begin{split}
&\hat p(\pi - a -t) + \hat q(\pi - b -t) = 0, \quad
\\[2mm]
- \hat p(t+a - \pi) -& \hat p(\pi - t + a)
 + \hat p(\pi - a+t)  -  \\[2mm]
  -&\hat q(t+b - \pi)  - \hat q(\pi - t + b)  + \hat q(\pi - b+t)=0, \quad
\end{split} \right\}
\end{equation}
where $t \in [0, \pi].$
Considering the first equality in~\eqref{u01} for $t \in [0, \pi - a],$ after the substitution  $z = \pi - a - t,$ we obtain
\begin{equation}\label{eq pq}
\hat p(z) + \hat q(z + a - b) = 0, \quad z \in [0, \pi - a].
\end{equation}
For definiteness, we assume that $p,$ $q,$ $\tilde p,$ and $\tilde q$ are zero on $[0, b].$
Considering the second equality in~\eqref{u01} for $t \in [0, a],$  taking into account that $\hat p = \hat q = 0$ on $[0, b],$ we get $\hat p(\pi - a+t) + \hat q(\pi - b +t)=0.$ The substitution $z = \pi - a - t$ yields equality~\eqref{eq pq} for $z \in [\pi - a, \pi].$ Thus, we arrive at the formula
\begin{equation} \label{pq1}
\hat p(z) + \hat q(z + a - b) = 0, \quad z \in [0, \pi].
\end{equation}
Analogously, the second equality in~\eqref{u01} for $t \in [b, \pi]$ yields $\hat p(\pi - t + a) + \hat q(\pi -t +b) = 0,$ and, after a substitution, we obtain
\begin{equation}\label{pq2}
\hat p(z + a - b) + \hat q(z) = 0, \quad z \in [b ,\pi].
\end{equation}
Since $\hat p = \hat q = 0$ on $[0, b],$ considering $b \le z \le \min(2b - a, \pi)$ in formulae~\eqref{pq1} and~\eqref{pq2}, we arrive at the equalities $\hat p(z) = \hat q(z) = 0.$ Repeating these arguments, by induction, we prove that
$$\hat p = \hat q = 0, \quad b + (k-1)(b-a) \le z \le \min(b + k(b-a), \pi), \quad k =1, \ldots, n,$$
where $n \in \mathbb N$ is such smallest number that $b + n(b - a) \ge \pi.$ These equalities mean that $\hat p = \hat q = 0$ on $[b, \pi],$ and the statement of the theorem is proved. The case when $p,$ $q,$ $\tilde p,$ and $\tilde q$ equal zero on $[a, \pi]$ is considered analogously.
\end{proof}

\section{Regularized trace formulae} \label{traces sec}

For $n \ge 1,$  put
$$a_{n0} =  \frac2\pi \sin na \int_0^\pi  \sin nt \, p(t) \, dt, \quad a_{n1} =  \frac2\pi \cos \Big(n-\frac12\Big)a \int_0^\pi  \cos \Big(n-\frac12\Big)t \, p(t) \, dt,$$
$$b_{n0} = \frac2\pi \sin nb \int_0^\pi  \sin nt \, q(t) \, dt, \quad b_{n1} = \frac2\pi \cos \Big(n-\frac12\Big)b \int_0^\pi  \cos \Big(n-\frac12\Big)t \, q(t) \, dt$$
and introduce numbers $s_{nj} = a_{nj} + b_{nj},$ $j=0,1.$

\begin{theorem} \label{th traces}
Let $p, q \in L_1[0, \pi]$ and $j=0,1.$ The series $\sum_{n=1}^\infty \left(\lambda_{nj} - \Big(n - \frac{j}{2}\Big)^2\right)$ converges if and only if the series $\sum_{n=1}^\infty s_{nj}$ converges. In the case of convergence, formula~\eqref{traces} holds.
\end{theorem}
For the proof, we need the following lemma.
\begin{lemma} \label{Riemann-Lebesgue}
Consider real numbers $c < d$ and a function $f \in L_1[c, d].$ Then,
\begin{equation} \label{2 int}
\int_c^d e^{i \rho (t-c)} f(t) \, dt = o\big(e^{|\tau|(d-c)}\big), \quad \int_c^d e^{i \rho (d-t)} f(t) \, dt = o\big(e^{|\tau|(d-c)}\big), \quad \rho \to \infty,
\end{equation}
where $\tau = \mathrm{Im}\, \rho.$
\end{lemma}
\begin{proof}
%
By replacing the variable of integration, we reduce the both equalities in~\eqref{2 int} to the formula
\begin{equation}\label{m1}
I(\rho) := \int_0^z e^{i\rho t} g(t) \, dt = o\big(e^{|\tau|z}\big), \quad \rho \to \infty,
\end{equation}
with $z = d - c > 0$ and $g \in L_1[0, z].$
Now, we prove~\eqref{m1}. For an arbitrary $\varepsilon > 0,$ there exists a continuously differentiable function $\tilde g \in C^{(1)}[0, z]$ such that $\int_0^z|\tilde g(t) - g(t)|\, dt < \frac{\varepsilon}{2}.$ Denoting $ I_1(\rho) = \int_0^{z}e^{i \rho t} \tilde g(t) \, dt,$ we get
\begin{equation} \label{m2}
\left|\int_0^{z}e^{i\rho t} g(t) \, dt\right| \le \int_0^{z}\big|e^{i\rho t}\big| \, |\tilde g(t) - g( t)| \, dt + |I_1(\rho)| \le   e^{|\tau|z}\frac{\varepsilon}{2}+ |I_1(\rho)|.
\end{equation}
Integrating by parts $I_1(\rho),$ we get
$$|I_1(\rho)| \le  |\rho|^{-1} M_\varepsilon e^{|\tau|z}, \quad M_\varepsilon = 2 \sup_{t \in [0, z]} |\tilde g(t)| + \int_0^{\pi - a} |\tilde g'(t)| \, dt.$$
For sufficiently large $|\rho| > 2 \varepsilon^{-1} M_\varepsilon,$ we have $|I_1(\rho)| \le e^{|\tau|z}\frac{\varepsilon}{2}.$ Combining the latter estimate with~\eqref{m2}, for an arbitrary $\varepsilon > 0,$ we obtain
$$|I(\rho)| \le  \varepsilon e^{|\tau|z}, \quad |\rho| > 2\varepsilon^{-1} M_\varepsilon,$$
which means~\eqref{m1}.
\end{proof}

\begin{proof}
For definiteness, we assume that $j=0;$ computations for $j=1$ are analogous. Denote
$$S(\lambda) = \frac{\sin \rho \pi}{\rho}, \quad \Gamma_N = \Big\{ \lambda \in \mathbb C \colon |\lambda | = \Big(N + \frac12\Big)^2  \Big\}, \;N \in \mathbb N.$$ Then,
\begin{equation} \label{IN}
I_N := \sum_{n=1}^N (\lambda_{n0} - n^2) = \frac{1}{2 \pi i} \mathop{\int}_{\Gamma_N} \lambda \left( \ln \frac{\Delta_0(\lambda)}{S(\lambda)}\right)' \, d\lambda.
\end{equation}
For $\lambda \in \Gamma_N,$ we estimate
\begin{equation*} \label{jm}
\frac{\Delta_0(\lambda)}{S(\lambda)} = 1 + f(\lambda), \quad
f(\lambda) := \frac{A_{00}(\lambda)+A_{01}(\lambda)+B_{0}(\lambda)}{S(\lambda)}.
\end{equation*}
By the standard way (see e.g.~\cite{KuM_yurko}) one can prove that\begin{equation} \label{b0}
|S(\lambda)| \ge C \frac{e^{|\tau|\pi}}{|\rho|}, \quad  \lambda \in \Gamma_N.
\end{equation}
Since $\sin \rho \xi = \frac{e^{i \rho \xi} - e^{-i \rho \xi}}{2i}$ and $\cos \rho \xi = \frac{e^{i \rho \xi} + e^{-i \rho \xi}}{2},$ we can apply Lemma~\ref{Riemann-Lebesgue} to each integral in~\eqref{def A} and~\eqref{Bj}.
Then, we get
\begin{equation} \label{a00}
A_{00}(\lambda) = o\Big(\frac{e^{|\tau|\pi}}{\rho^2}\Big), \quad A_{01}(\lambda) = o\Big(\frac{e^{|\tau|\pi}}{\rho^2}\Big),\quad B_0(\lambda) = o\Big(\frac{e^{|\tau|\pi}}{\rho^3}\Big).
\end{equation}

Thus, due to~\eqref{b0} and \eqref{a00}, we have $f(\lambda) =  o(\rho^{-1})$ for $\lambda \in \Gamma_N,$ and for large $N,$ the increment of the argument $\Delta_0(\lambda)/S(\lambda)$ on the contour $\Gamma_N$ equals~$0.$ Integrating~\eqref{IN} by parts, we arrive at the formula
$$I_N = -\frac{1}{2 \pi i}\mathop{\int}_{\Gamma_N} \ln \Big(1 + f(\lambda)\Big) \, d\lambda. $$ Applying the Taylor expansion to $\ln (1 + f(\lambda)),$ taking into account that $f(\lambda) =  o(\rho^{-1})$ and $B_0(\lambda)/S(\lambda) = o(\rho^{-2}),$ we get
\begin{multline*}
I_N = -\frac{1}{2 \pi i}\mathop{\int}_{\Gamma_N} \left(\frac{A_{00}(\lambda)}{S(\lambda)}+\frac{A_{01}(\lambda)}{S(\lambda)} + o\Big(\frac{1}{\rho^2}\Big)\right)\, d\lambda =\\
= - \sum_{n=1}^N \mathop{\mathrm{Res}}_{\lambda = n^2} \frac{A_{00}(\lambda)}{S(\lambda)} - \sum_{n=1}^N \mathop{\mathrm{Res}}_{\lambda = n^2} \frac{A_{01}(\lambda)}{S(\lambda)} + o(1).
\end{multline*}
Computing $\displaystyle\mathop{\mathrm{Res}}_{\lambda = n^2} \, \Big(A_{00}(\lambda)/S(\lambda)\Big) = -a_{n0}$ and $\displaystyle\mathop{\mathrm{Res}}_{\lambda = n^2} \, \Big(A_{01}(\lambda)/S(\lambda)\Big) = -b_{n0},$ we arrive at the formula $I_N = \sum_{n=1}^N s_{n0} + o(1).$ As $N$ tends to $\infty,$ we obtain the statement of the theorem.
\end{proof}

The trigonometric systems of the functions $\{ \sqrt\frac\pi 2 \sin nt\}_{n \ge 1}$ and $\{ \sqrt\frac\pi 2 \cos (n-\frac12)t\}_{n \ge 1}$ are orthonormal bases in $L_2(0, \pi),$ being the systems of the eigenfunctions of the unperturbed operator $-y''$ with the boundary conditions~\eqref{bc} for $j=0$ and $j=1,$ respectively.
The series $\sum_{n=1}^\infty a_{nj}$ is the Fourier series of the function $p$ at the point $a,$ while the series $\sum_{n=1}^\infty b_{nj}$ is the Fourier series of the function $q$ at the point $b.$ For trigonometric series, there are known several convergence tests for the Fourier series of a function $f$ at a separate point.
In particular, it suffices to claim the absolute continuity of $f$ at the vicinity of the point so that the series converges at the value of $f$ at this point.
If $p \in AC[a-\varepsilon, a+\varepsilon]$ and $q \in AC[b-\varepsilon, b+\varepsilon]$ for some $\varepsilon > 0,$ then $\sum_{n=1}^\infty a_{nj} = p(a),$ $\sum_{n=1}^\infty b_{nj}=q(b),$ and
$$\sum_{n=1}^\infty \left(\lambda_{nj} - \Big(n-\frac{j}{2}\Big)^2\right) = p(a) + q(b),$$
which agrees with the results of the previous works~\cite{bon-traces,dob-traces}.
Note that the convergence of the series in~\eqref{traces} may hold even if the series  $\sum_{n=1}^\infty a_{nj}$ and $\sum_{n=1}^\infty b_{nj}$ diverge.

\begin{example}
Let $j=0,$ $a = \frac\pi3,$ and $b = \pi - a.$
We take the functions $p, q \in L_2(0, \pi)$ such that $$\int_0^\pi p(t) \sin nt \, dt =  \frac{\pi}{2n} \mathrm{sgn} \, \big(\sin na \big), \; n \ge 1, \quad q(t) = -p(\pi - t).$$
Then, $a_{n0} = \frac{\sqrt3}{2n}$ when $n$ is not a multiple of $3,$ and  $a_{n0} = 0$ otherwise. Herewith,
$$ \sum_{n = 1}^\infty a_{n0} = \frac{\sqrt3}{2}\sum_{k=1}^\infty \left(\frac{1}{3k - 2} + \frac{1}{3k-1}\right) > \frac{\sqrt3}{2} \sum_{k=1}^\infty \frac{1}{3k},$$
and the series diverges. On the other hand, applying the equalities $b = \pi -a$ and $q(t) = -p(\pi - t),$ we obtain $b_{n0} = -a_{n0}$ and  $s_{n0} = 0,$ $n \ge 1.$
Thus, the both series $\sum_{n=1}^\infty a_{n0}$ and $\sum_{n=1}^\infty b_{n0}$ diverge, but $\sum_{n=1}^\infty s_{n0}=0.$ By virtue of Theorem~\ref{th traces}, we arrive at the formula
$$ \sum_{n=1}^\infty (\lambda_{n0} - n^2) = 0.$$
The latter equality holds whenever $b = \pi -a$ and $q(t) = -p(\pi - t).$
\end{example}

\medskip
\noindent{\bf Acknowledgments.} This work was supported financially by project no.~24-71-10003 of the Russian Science Foundation, see~\url{ https://rscf.ru/en/project/24-71-10003/}.

\section*{Appendix: the proof of formula~\eqref{rep}}
For definiteness, we consider the case $j=0.$ Expanding the determinant into the sum by the third column and then into the sums by the second column, we get
 $$\Delta_0(\lambda) = D_{00} +D_{01} + D_{10} +D_{11}, \quad D_{00} := \begin{vmatrix}
 \displaystyle\frac{\sin \rho \pi}{\rho} &  \displaystyle 0 & 0 \\[3mm]
 \displaystyle\frac{\sin \rho a}{\rho} & -1 & 0 \\[3mm]
 \displaystyle\frac{\sin \rho b}{\rho} & 0 & -1
\end{vmatrix},$$
$$
D_{01} := \begin{vmatrix}
 \displaystyle\frac{\sin \rho \pi}{\rho} & \displaystyle \int_0^\pi \frac{\sin \rho(\pi - t)}{\rho} p(t) \,dt& 0 \\[3mm]
 \displaystyle\frac{\sin \rho a}{\rho} & \displaystyle\int_0^a \frac{\sin \rho(a - t)}{\rho} p(t) \, dt & 0 \\[3mm]
 \displaystyle\frac{\sin \rho b}{\rho} & \displaystyle\int_0^b \frac{\sin \rho(b - t)}{\rho} p(t) \, dt & -1
\end{vmatrix},
\quad D_{10} := \begin{vmatrix}
 \displaystyle \frac{\sin \rho \pi}{\rho} & 0 & \displaystyle \int_0^\pi \frac{\sin \rho(\pi - t)}{\rho} q(t) \, dt \\
 \displaystyle\frac{\sin \rho a}{\rho} & -1 & \displaystyle\int_0^a \frac{\sin \rho(a - t)}{\rho} q(t) \, dt \\
 \displaystyle\frac{\sin \rho b}{\rho} & 0 & \displaystyle\int_0^b \frac{\sin \rho(b - t)}{\rho} q(t) \, dt
\end{vmatrix}, $$
$$
D_{11} = \begin{vmatrix}
 \displaystyle\frac{\sin \rho \pi}{\rho} & \displaystyle \int_0^\pi \frac{\sin \rho(\pi - t)}{\rho} p(t) \,dt & \displaystyle \int_0^\pi \frac{\sin \rho(\pi - t)}{\rho} q(t) \, dt \\[3mm]
 \displaystyle\frac{\sin \rho a}{\rho} & \displaystyle\int_0^a \frac{\sin \rho(a - t)}{\rho} p(t) \, dt & \displaystyle\int_0^a \frac{\sin \rho(a - t)}{\rho} q(t) \, dt\\[3mm]
 \displaystyle\frac{\sin \rho b}{\rho} & \displaystyle\int_0^b \frac{\sin \rho(b - t)}{\rho} p(t) \, dt &\displaystyle\int_0^b \frac{\sin \rho(b - t)}{\rho} q(t) \, dt
\end{vmatrix}.
$$
Clearly, $D_{00} = \frac{\sin \rho \pi}{\rho},$ which gives us the first term in~\eqref{rep}. Consider $D_{01}:$
\begin{multline*}
D_{01} =  \frac{\sin \rho a}{\rho} \int_0^\pi \frac{\sin \rho(\pi - t)}{\rho} p(t)\,dt
-\frac{\sin \rho \pi}{\rho}\int_0^a \frac{\sin \rho(a - t)}{\rho} p(t) \, dt= \\[2mm]
=  \frac{\sin \rho a}{\rho} \int_a^\pi \frac{\sin \rho(\pi - t)}{\rho} p(t)\,dt + \frac{1}{\rho^2}\int_0^a \big[\sin \rho a \sin \rho(\pi - t)-\sin \rho \pi \sin \rho(a - t)\big] p(t) \,dt.
\end{multline*}
Further, we need the trigonometric formulae
\begin{equation} \label{trig form}
\left.\begin{array}{c}
\sin \alpha \sin (\beta - \gamma) - \sin \beta \sin (\alpha - \gamma)= \sin \gamma \sin (\beta - \alpha), \\[2mm]
\sin \alpha \cos (\beta - \gamma) - \cos \beta \sin (\alpha - \gamma)= \sin \gamma \cos (\beta - \alpha).
\end{array}\right\}
\end{equation}
Applying the first formula from~\eqref{trig form}, we obtain
$$
\sin \rho a \sin \rho(\pi - t)-\sin \rho \pi \sin \rho(a - t) = \sin \rho t \sin \rho(\pi - a),
$$
and we arrive at the equality $D_{01} = A_{00}(\lambda).$ Analogously, $D_{10} = A_{01}(\lambda).$

Let us put $B_0(\lambda)=D_{11}$ and bring $B_0(\lambda)$ to the needed form~\eqref{Bj}. We expand subsequently the determinant $D_{11}$ into the sums by the second and the third columns, splitting the integrals by the limits: $\int_0^\pi = \int_0^a + \int_a^b + \int_b^\pi$ and $\int_0^b = \int_0^a + \int_a^b.$ Let us write out the term that contains the integrals of the functions $p$ and $q$ only on the segments $[0, a]:$
\begin{equation} \label{summand1}
\begin{vmatrix}
\displaystyle \frac{\sin \rho \pi}{\rho} & \displaystyle \int_0^a \frac{\sin \rho(\pi - t)}{\rho} p(t) \,dt & \displaystyle \int_0^a \frac{\sin \rho(\pi - t)}{\rho} q(t) \, dt \\[3mm]
\displaystyle \frac{\sin \rho a}{\rho} & \displaystyle\int_0^a \frac{\sin \rho(a - t)}{\rho} p(t) \, dt & \displaystyle\int_0^a \frac{\sin \rho(a - t)}{\rho} q(t) \, dt\\[3mm]
\displaystyle \frac{\sin \rho b}{\rho} & \displaystyle\int_0^a \frac{\sin \rho(b - t)}{\rho} p(t) \, dt &\displaystyle\int_0^a \frac{\sin \rho(b - t)}{\rho} q(t) \, dt
\end{vmatrix}.
\end{equation}  Subtract the second row multiplied by $\cos \rho(\pi - a)$ from the first row, and subtract the second row multiplied by $\cos \rho(b - a)$ from the third row.
Denoting $x_1 = \frac{\sin \rho (\pi - a)}{\rho}$ and $x_3 = \frac{\sin \rho (b - a)}{\rho},$ we get
$$\begin{vmatrix}
x_1 \cos \rho a & x_1 \displaystyle \int_0^a \cos \rho(a-t) p(t) \,dt & x_1\displaystyle \int_0^a \cos \rho(a-t) q(t) \,dt \\[3mm]
\displaystyle\frac{\sin \rho a}{\rho} & \displaystyle\int_0^a \frac{\sin \rho(a - t)}{\rho} p(t) \, dt & \displaystyle\int_0^a \frac{\sin \rho(a - t)}{\rho} q(t) \, dt\\
x_3 \cos \rho a & x_3\displaystyle \int_0^a \cos \rho(a-t) p(t) \,dt & x_3\displaystyle \int_0^a \cos \rho(a-t) q(t) \,dt
\end{vmatrix} = 0,$$
since the first and the third row are linearly dependent. Thus, determinant~\eqref{summand1} equals zero. This means that the term containing as factors  the integral of the function $p$ on $[0, a]$ and the integral of the function $q$ on $[0, a]$ is absent in~\eqref{Bj}.

Let us write out the term containing as factors the integral of $p$ on $[0, a]$ and the integral of $q$ on $[a, b]:$
$$\begin{vmatrix}
 \displaystyle \frac{\sin \rho \pi}{\rho} & \displaystyle \int_0^a \frac{\sin \rho(\pi - t)}{\rho} p(t) \,dt & \displaystyle \int_a^b \frac{\sin \rho(\pi - t)}{\rho} q(t) \, dt \\[3mm]
 \displaystyle \frac{\sin \rho a}{\rho} & \displaystyle\int_0^a \frac{\sin \rho(a - t)}{\rho} p(t) \, dt & 0\\[3mm]
 \displaystyle \frac{\sin \rho b}{\rho} & \displaystyle\int_0^a \frac{\sin \rho(b - t)}{\rho} p(t) \, dt &\displaystyle\int_a^b \frac{\sin \rho(b - t)}{\rho} q(t) \, dt
\end{vmatrix}.
 $$
Subtract the third row multiplied by $\cos \rho(\pi - b)$ from the first row, then subtract the second row multiplied by $\cos \rho(b-a)$ from the third row. We get
\begin{multline*}
\frac{\sin \rho(\pi - b)}{\rho^3} \begin{vmatrix}
\cos \rho b & \displaystyle \int_0^a \cos \rho(b-t) p(t) \,dt & \displaystyle \int_a^b \cos \rho(b-t) q(t) \, dt \\
\sin \rho a & \displaystyle\int_0^a \sin \rho(a - t) p(t) \, dt & 0\\
\sin \rho (b-a) \cos \rho a & \sin \rho(b-a)\displaystyle\int_0^a \cos \rho(a-t) p(t) \, dt &\displaystyle\int_a^b \sin \rho(b - t) q(t) \, dt
\end{vmatrix} =
\\=\frac{\sin \rho(\pi - b)}{\rho^3} \Bigg(\int_a^b \cos \rho(b-\xi) q(\xi)\, d\xi \int_0^a\big[ \sin \rho a  \sin \rho(b-t) - \sin \rho b \sin \rho(a-t) \big] p(t)\, dt -
\\-\int_a^b \sin \rho(b-\xi) q(\xi)\, d\xi \int_0^a \big[ \sin \rho a \cos \rho(b-t)  - \cos \rho b \sin \rho(a-t)\big]p(t)\, dt\Bigg).\end{multline*}
Applying formulae~\eqref{trig form} to the terms in the square brackets, we arrive at the expression
\begin{multline*}
\frac{\sin \rho(\pi - b)}{\rho^3} \int_0^a \sin \rho t p(t) \, dt  \int_a^ b \big( \sin \rho(b-a)\cos \rho(b-\xi)  - \cos \rho(b-a)  \sin \rho(b-\xi) \big)q(\xi)\, d\xi =\\= \frac{\sin \rho(\pi - b)}{\rho^3} \int_0^a \sin \rho t p(t) \, dt \int_a^ b \sin \rho(\xi-a) q(\xi)\, d\xi,\end{multline*}
which gives the first term in the first pair of brackets in~\eqref{Bj}. The other terms in the expansion of the determinant $D_{11}$ are considered analogously.
\end{document}